\documentclass[12pt,a4paper,reqno]{amsart}

\usepackage{a4wide}
\usepackage{amssymb,amsmath,amsthm,mathrsfs}
\usepackage{graphicx}
\usepackage{float}
\usepackage{colortbl}
\usepackage{enumerate}
\usepackage{epstopdf}

\theoremstyle{plain}% default

\newtheorem{proposition}{Proposition}

\theoremstyle{definition}
\theoremstyle{remark}
\newtheorem{remark}{Remark}%[section]

\def\R{\ensuremath{\mathbb R}}
\def\N{\ensuremath{\mathbb N}}

\def\e{{\ensuremath{\rm e}}}

\def\L{\ensuremath{\mathcal L}}

\def\p{\ensuremath{\mathbb P}}

\def\t{\ensuremath{t}}

\def\eps{\varepsilon}

\numberwithin{equation}{section}

\begin{document}

\title{Extreme Value laws for dynamical systems under  observational noise }% Force line breaks with \\

\author[D. Faranda]{Davide Faranda}
\address{Laboratoire SPHYNX, Service de Physique de l'Etat Condens\'e, DSM,
CEA Saclay, CNRS URA 2464, 91191 Gif-sur-Yvette, France
}
\email{davide.faranda@cea.fr}

\author{ Sandro Vaienti}
\address{UMR-7332 Centre de Physique Th\'{e}orique, CNRS, Universit\'{e}
d'Aix-Marseille I, II, Universit\'{e} du Sud, Toulon-Var and FRUMAM,
F\'{e}d\'{e}ration de Recherche des Unit\'{e}s des Math\'{e}matiques de Marseille,
CPT Luminy, Case 907, F-13288 Marseille CEDEX 9}
\email {vaienti@cpt.univ-mrs.fr}



\begin{abstract}
In this paper we prove the existence of Extreme Value Laws for  dynamical systems perturbed by instrument-like-error, also called observational noise. An orbit perturbed with observational noise  mimics the behavior of  an instrumentally recorded  time series. Instrument characteristics  - defined as precision and accuracy -  act both by truncating  and randomly displacing the real value of a measured observable.  Here we analyze both these effects from a theoretical and numerical point of view. First we show that classical extreme value laws can be found for orbits of dynamical systems perturbed with   observational noise.  Then we present  numerical experiments to support the theoretical findings and give an indication of the order of magnitude of the instrumental perturbations which  cause relevant deviations from the extreme value laws observed in deterministic dynamical systems. Finally, we show that the  observational noise preserves the structure of the deterministic attractor. This goes against the  common assumption that random transformations cause the orbits asymptotically fill the ambient space with a loss of information about any   fractal structures present on the attractor.

\end{abstract}

\maketitle

%\tableofcontents

\section{Introduction}

\label{dad}
In two previous works \cite{22, 30}, we investigated the persistence of Extreme Value Laws (EVLs) whenever a dynamical system is perturbed throughout random transformations. We  considered  an i.i.d.  stochastic process $(\omega_k)_{k\in N}$ with values in the measurable space $Q_{\varepsilon}$ and with probability distribution $\theta_{\varepsilon}$. After associating to each $\omega\in Q_{\varepsilon}$ a map $T_{\omega}$ acting on the measurable space  $\Omega$ into itself,    we considered the random orbit starting from the point $x$ and generated by the realization $\underline{\omega}_n=(\omega_1,\omega_2,\cdots,\omega_n)$: $$T_{\underline{\omega}_n}:=T_{\omega_n}\circ\cdots\circ T_{\omega_1}(x).$$  Here, the transformations $T_{\omega}$ should be considered  close to each other and the suitably rescaled scalar parameter $\varepsilon$ is the strength of such a distance.  We could therefore  define a Markov process  $\mathcal X_{\varepsilon}$ on $\Omega$ with transition function
\begin{equation}\label{gre}
 P (x, A)=\int_{Q_{\varepsilon}}\mathbf{1}_{A}(T_{\omega}(x))d\theta_{\varepsilon}(\omega),
 \end{equation}
where $A\in\Omega$ is a measurable set,  $x\in \Omega$ and $\mathbf{1}_{A}$ is the indicator function of a set $A$. A probability measures $\mu_{\varepsilon}$ is called a \textit{stationary measure} if for any measurable $A$ we have:
$$
\mu_{\varepsilon}(A)=\int_{\Omega} P_{\varepsilon}(x,A)d\mu_{\varepsilon}(x)
.$$
We call it an absolutely continuous stationary measure ({\em acsm}), if it has a density with respect to the Lebesgue measure whenever $\Omega$ is a metric space. \\

In this work we consider a different type of perturbation, the   observational noise, which consists in replacing the orbit of the point $x\in \Omega$ at time $i$, namely $T^ix,$ with $T^ix +\omega_i.$ There are several physical motivations to investigate the behavior of this kind of perturbation. In fact, as Lalley and Noble wrote in \cite{LN}:

{\em  "...In this model our observations take the
form $y_i = T^ix+ \omega_i$, where $\omega_i$ are independent, mean zero random vectors. In
contrast with the dynamical noise model (e.g; the  random transformations), the noise does not interact with the
dynamics: the deterministic character of the system, and its long range dependence,
are preserved beneath the noise. Due in part to this dependence, estimation in the
 observational noise  model has not been broadly addressed by statisticians, though
the model captures important features of many experimental situations."}

Judd \cite{J}, quoted in \cite{SD}, also pointed out that:

{\em "...the reality is that many physical
systems are indistinguishable from deterministic systems, there is no apparent
small dynamic noise, and what is often attributed as such is in fact model error."}

Moreover, a system contaminated by the  observational noise  raises the natural  and practical  question whether it would be possible to recover the original signal, in our case the deterministic orbit $\{T^ix\}_{i\ge 1}$. In the last years a few techniques have been proposed for such a  noise reduction \cite{KS}: we remind here   the remarkable Schreiber-Lalley method \cite{LL1, LL2,  SS1, SS2}, which provides a very consistent algorithm  to perform the noise reduction when the underlying deterministic dynamical system has strong hyperbolic properties. Another interesting work shows that in the computation of some statistical quantities, the dynamical noise corresponding to random transformations could be considered as an  observational noise  with the Cauchy distribution  \cite{SMR}. Finally, the paper \cite{34} proves concentration inequalities for systems perturbed by  observational noise.\\

The present work try to re-frame the previous findings in terms of extreme value theory (EVT) by adding a further motivation driven by the applicability of the whole  EVT  for dynamical systems to experimental data. It should be a general praxis to check the role of instrument-like-perturbations before applying dynamical systems techniques to experimental datasets. In this sense, the dynamical systems considered in this paper share several properties with observed time series, as the observational noise acts exactly as a physical instrument.   The goal is to exploit the recent  advancements of the  EVT for dynamical systems to define in a more rigorous way the extremes of time series. A   relevant example of a successful application of the theory presented in this paper to experimental datasets is given in \cite{31b}, where temperature data are analyzed with the algorithmic procedure   presented in Section 4.2.  More specifically,  our interest is to understand which way the results obtained on deterministic dynamical systems are  altered by the addition of observational noise and in which cases one can recover   classical EVLs.  We start the discussion by summarizing the main findings of the EVT for dynamical systems.\\

The first rigorous  mathematical approach to EVT in dynamical systems goes back to the pioneer paper by Collet \cite{12}. Important contributions have successively been given in  \cite{13}, \cite{14}, \cite{15} and in \cite{16}. Here we briefly  recall the main findings  deferring to the previous papers for the full demonstrations.\\

  Let us consider a dynamical system $(\Omega, {\mathcal B}, \nu, T)$, where $\Omega$ is the
    invariant set in some manifold, usually $\mathbb{R}^d$, ${\mathcal B}$ is the Borel $\sigma$-algebra, $T:\Omega\rightarrow \Omega$ is a measurable map
      and $\nu$ a  probability $T$-invariant Borel measure.\\
In order to adapt the EVT to dynamical systems, we follow \cite{13}, by considering the stationary stochastic process $X_0,X_1,...$  given by:

\begin{equation}
X_m(x)=w (\mbox{dist}(T^m x, z)) \qquad \forall m \in \mathbb{N},
\label{sss}
\end{equation}

where 'dist' is a distance on the ambient space  $\Omega$, $z$ is a given point and $w$ is a suitable function which will be specified later.  This particular functional form has been introduced first by Collet \cite{12} and allows for a direct connection between recurrence properties around a point of the phase space $z$ and the existence of EVLs.  The object of interest is the distribution of $\mathbb{P}(M_m\le u_m)$, where $M_m:=\max\{ X_0, \cdots, X_{m-1}\}$; we say that we have an EVL for $M_m$ if there is a non-degenerate distribution function \ $H:\R\to[0,1]$ with $H(0)=0$ and,  for every $\tau>0$, there exists a sequence of levels $u_m=u_m(\tau)$, $m=1,2,\ldots$,  such that
\begin{equation}
\label{eq:un}
  m\,\p(X_0>u_m)\to \tau,\;\mbox{ as $m\to\infty$,}
\end{equation}
and for which the following limit holds:
$$
\p(M_m \le u_m)\rightarrow 1-H(\tau), \mbox{as} \ m\rightarrow \infty
$$

The motivation for using a normalizing sequence $u_m$ satisfying \eqref{eq:un} comes from the case when $X_0, X_1,\ldots$ are independent and identically distributed (i.i.d). In this setting, it is clear that $\p(M_m\leq u)= (F(u))^m$, being $F(u)$ the cumulative distribution function for the variable $u$. Hence, condition \eqref{eq:un} implies that
\[
\p(M_m\leq u_m)= (1-\p(X_0>u_m))^m\sim\left(1-\frac\tau m\right)^m\to\e^{-\tau},
\]
as $m\to\infty$.  Note that in this case $H(\tau)=1-\e^{-\tau}$ is the standard exponential distribution function. By choosing  the sequence $u_m=u_m(y)$ as one parameter families like $u_m=y/a_m+b_m$, where $y\in \mathbb{R}$ and $a_m>0,$ for all $m\in \mathbb{N}$ and $w$ as above,   whenever the variables $X_i$ are i.i.d., if for some constants $a_m>0, \ b_m,$ we have $\mathbb{P}(a_m(M_m-b_m)\leq y)\rightarrow G(y)$.  When the convergence occurs at continuity points of $G$   ($G$ is non-degenerate) then $G_m$ converges to one of the three EVLs  rewritable in terms of the Generalized Extreme Value (GEV) distribution as:

\begin{equation}
G(y; \kappa)=\exp\left\{ [1+{\kappa}y]^{-1/{\kappa}}\right\}.
\label{cumul}
\end{equation}

Here ${\kappa} \in \mathbb{R}$ is the shape parameter also called the tail index: when ${\kappa} \to 0$, the distribution corresponds to a Gumbel EVL; when the tail index is positive, it corresponds to a Fr\'echet EVL; when ${\kappa}$ is negative, it corresponds to a Weibull EVL.  The EVL obtained depends on the kind of observable chosen. In particular, in \cite{12,13} the authors have shown that, once taken the observable:

\begin{equation}
w(y)= -\log(y),
\label{g1}
\end{equation}

one gets a Gumbel EVL, here $y=\mbox{dist}(T^m x, z)$. In the next section we  prove the existence of Gumbel law for the maps perturbed with observational noise. It is in fact possible to  introduce other observables  than the one specified above in order to get  convergence towards  Frechet and Weibull EVLs. However, for any  choice  different from $w(y)=-\log(y)$,  the tail index  depends  on the local property of the measure.
%We assume that $\varphi$ achieves a global maximum at $z\in\mathcal M$, for every $u<\varphi(z)$ but sufficiently close to $\varphi(z)$, the event $\{x\in\mathcal M:\; \varphi(x)>u\}=\{X_0>u\}$ corresponds to a topological ball ``centred'' at $z$ and,
For every sequence $(u_m)_{m\in\N}$ satisfying \eqref{eq:un} we define:
\begin{equation}
\label{def:Un}
U_m:=\{X_0>u_m\}.
\end{equation}
%is a nested sequence of sets such that
%\begin{equation}
%\label{def:zeta}
%\bigcap_{n\in\N} U_n=\{z\}.
%\end{equation}

When $X_0,X_1,X_2,\ldots$ are not independent, the standard exponential law still applies under some conditions on the dependence structure. These conditions are the following:

\textbf{Condition}[$D_2(u_m)$]\label{cond:D2} We say that $D_2(u_m)$ holds for the sequence $X_0,X_1,\ldots$ if for all $\ell,t$
and $m$,
\begin{equation}\label{D1}
|\p\left(X_0>u_m\cap
  \max\{X_{t},\ldots,X_{t+\ell-1}\leq u_m\}\right)-
 \p(X_0>u_m)
  \p(M_{\ell}\leq u_m)|\leq \gamma(m,t),
\end{equation}
where $\gamma(m,t)$ is decreasing in $t$ for each $m$ and
$m\gamma(m,t_m)\to0$ when $m\rightarrow\infty$ for some sequence
$t_m=o(m)$.

Now, let $(k_m)_{m\in\N}$ be a sequence of integers such that
\begin{equation}
\label{eq:kn-sequence-1}
k_m\to\infty\quad \mbox{and}\quad  k_m t_m = o(m).
\end{equation}

\textbf{Condition}[$D'(u_m)$]\label{cond:D'} We say that $D'(u_m)$
holds for the sequence $X_0, X_1, X_2, \ldots$ if there exists a sequence $(k_m)_{m\in\N}$ satisfying \eqref{eq:kn-sequence-1} and such that
\begin{equation}
\label{eq:D'un}
\lim_{m\rightarrow\infty}\,m\sum_{j=1}^{\lfloor m/k_m \rfloor}\p( X_0>u_m,X_j>u_m)=0.
\end{equation}

By following Freitas and Freitas \cite{FF}--Theorem 1, if conditions $D_2(u_m)$ and $D'(u_m)$ hold for  $X_0,X_1,X_2,\ldots,$ then there exists an EVL for $M_m$ and $H(\tau)=1-e^{-\tau}.$ \\ In the paper \cite{13}, Freitas, Freitas and Todd made the interesting observation that the extreme value laws are intimately related to the concept of local recurrence, in particular to the first hitting  time function in small sets. Their analysis has been brought to systems perturbed with random transformations in \cite{22}. In Section 3 we will show that these kind of results  hold also  for systems perturbed with observational noise, first by adapting the definition of first hitting time  and then by showing that it follows an exponential law tempered by the strength of the perturbation.\\

We remark that the analogy between extreme value laws and local recurrences is  possible for particular observables of the type $X_0(\cdot)=w(\mbox{dist}(\cdot,z)),$ where $z$ is a given point. As explained in \cite{FF}, the additional choice $w(y)=-\log(y),$ allows to get the Gumbel law (a direct proof of this fact is given after Proposition \ref{P2}), and moreover it brings information on the local structure of the invariant measure, as we will explain in a moment. This represents the main motivation for using   such    observable although a more detailed discussion and other motivations can be found in    \cite{30, 31a}. \\

We conclude this introduction  by stressing what we believe is an interesting and very general result. We will see that the probability $\mathbb{P}$ which we will use to rule out the distribution of the maxima is the product between the invariant measure $\nu$ and the measure of the noise $\theta$. Whenever one is able to prove the existence of an extreme value law for the process $X_m$ with the observable (\ref{g1}), then Proposition \ref{P2} shows how the two sequences $a_m$ and $b_m$ appearing in the affine choice for $u_m$ (see above), are related to the local behavior of the measure $\nu$ at a local scale given by the intensity of the noise. This local behavior is related to the fine structure of the measure $\nu$. We have therefore a useful tool to detect the fine geometric properties of the invariant measure by calibrating the normalizing sequence $u_m$ until we get the Gumbel law, and this at different scales for the noise.

\section{Recurrences for time series: a theoretical approach}

We now show how to adapt the EVT for orbits of dynamical systems   perturbed by instrument-like-error.

 Each instrument introduces an effect related  to the combined accuracy and precision of the measure  by replacing  the  real (unknown) value  with the biased  indicated by the instrument itself.
In general, if the real dynamics can be represented by the map $T$, what one observes is formally:

$$\varphi(i)=\operatorname{trunc}(T^ix+ \epsilon\xi_i, q),$$

where $\operatorname{trunc}(x,q) = \frac{\lfloor 10^q \cdot x \rfloor}{10^q}$ is the truncation introduced by the instrument precision and $\epsilon\xi_i$ is a random displacement from the real value. This displacement is what we have  defined as   observational noise. Here it is rewritten as  $T^ix+\epsilon \xi_i$, with  the parameter $\epsilon\in \mathbb{R}^+$.\\

Although the paper is dedicated to the  observational noise, for  completeness and in light of applications on time series, we want remark also some relevant properties of truncations.
The role of the truncation, important also for numerical computations, has been discussed in the book by Knuth \cite{17} and then analyzed, among others, by \cite{18,19,20}. On the $q$ digit, the truncation acts essentially as a random noise of variance $\sigma=10^{-q}$. If $q\gg 1$ the measure underlying   $\varphi(i)$ will match the one of the original dynamics given by the map $T$, if not the support will appear as a collection of Dirac's deltas, precluding the convergence to the GEV distribution  \cite{21}.  We defer to Section 4.1 for a numerical study on the truncation error  which should point out for which values of $q$ one should take truncations into account and whether $q \simeq 7$ (the common truncation corresponding to a double precision representation) is a good choice for representing the  properties of a deterministic dynamics.

We thus proceed to  show that   EVLs persist  for chaotic dynamical systems, whenever they are perturbed with the  observational noise \cite{33,34}.   The proof we give is reminiscent of that of Theorem D in  \cite{22}. We first point out that, in order to guarantee the stationarity of the random process involved in the distribution of the maxima, we need to evaluate the observable $w$ at the point $x+\xi$, where $x\in \Omega$ and $\xi$ is a random vector. For this reason and in order to avoid ambiguities, we will choose $\Omega$ as a torus. Another alternative would be to take $\Omega$ which is strictly sent into itself by $T$, $T\Omega\subset \Omega$, and with all the components of $\xi$ small enough (see, for instance, Proposition 4.5 in \cite{22}). This choice is often invoked for random transformation acting of bounded domains of $\mathbb{R}^n.$ The paragraph below contains the assumptions on the systems which allow us to prove our main results.\\

{\bf Assumption M}: We consider  maps $T$ defined  on the torus  $\Omega=\mathbb{T}^d$ with   norm $||\cdot||$ and satisfying:
 \begin{itemize}
 \item There exists a finite partition (mod-$0$) of $\Omega$ into open sets $Y_j, j=1,\cdots,p$, namely $\Omega=\cup_{j=1}^p \overline{Y_j}$, such that $T$ has a Lipschitz extension on the closure of each $Y_j$ with a uniform and strictly larger than $1$  Lipschitz constant $\eta$, $||T(x)-T(y)||\le \eta ||x-y||$, $\forall x,y \in \overline{Y_j}, j=1,\cdots,p.$
     \item  $T$ preserve a Borel probability measure $\nu$ which is also mixing with decay of correlations given by
\begin{equation}\label{DC}
\left|\int f\circ T^m h d\nu -\int f d\nu \int h d\nu \right|\le C ||h||_{\mathcal{B}}||f||_1 m^{-2}
\end{equation}
where the constant $C$ depends only on the map $T$, $||\cdot||_1$ denotes the $L^1_{\nu}$ norm with respect to $\nu$ and finally $\mathcal{B}$ is a Banach space included in $L^{\infty}_{\L},$ where $\L$ denotes the Lebesgue (Haar)  measure on $X$: the corresponding norm will be denoted with  $||\cdot||_{\infty}.$ We will also need  $\nu$  to  be equivalent to  $\L$ with density in $L^{\infty}_{\L}$.
\end{itemize}

{\bf Assumption N}: We consider a sequence  $\xi_i$ of i.i.d vector-valued random variables which take values in the hypershpere $S:=S^d\subset \mathbb{R}^d$ centered at $0$ and of radius $1$, $S:=\{u\in \mathbb{R}^d; \ ||u||\le 1\},$   and with  common distribution  $\theta$, which we choose absolutely continuous with density $\rho\in L^{\infty}_{\L}$, namely $d\theta(\xi) =\rho(\xi) d\L(\xi)$, with $\int_{S} \rho(\xi)  d\L(\xi)=1.$\footnote{Each $\xi$ is a vector with $d$ components; all these components are independent and distributed with common density $\rho'$; the  product of such marginals $\rho'$'s gives  $\rho$.}\\
%This means in particular that, whenever $\xi\in \mathbb{T}^d$ with $d>1$, then its components %$\xi^{(l)}, \ l=1,\cdots, d$ are i.i.d. random variables with distribution %$d\theta^{(l)}(\xi^{(l)})=\rho^{(l)}(\xi^{(l)}) d\Leb(\xi^{(l)})$, where %$\int_{\mathbb{T}}\rho^{(l)}(\xi^{(l)})d\Leb(\xi^{(l)})=1$, %$d\theta(\xi)=d\theta^{(1)}(\xi^{(1)})\cdots d\theta^{(d)}(\xi^{(d)})$ and finally %$\rho(\xi)=\rho^{(1)}(\xi^{(1)})\cdots \rho^{(d)}(\xi^{(d)}).$\\
\begin{remark}
The paper \cite{22} contains examples of endomorphisms verifying the Assumption M, in particular one-dimensional Rychlik maps endowed with bounded variation functions and multidimensional piecewise uniformly expanding maps endowed with quasi-H\"older observables. In order to get the decay of correlations (\ref{DC}), one needs more regularity for $T$, usually $C^{1+\alpha}.$ Our next proofs crucially depend on the decay against $L^1_{\nu}$ functions \footnote{Actually, we will stress  in Section 3 that what is really needed is the $L^1_{\nu}$ property for characteristic functions which, for some systems,  is  easier to show.}; we believe that one could weaken such   assumption and   extend the theory to invertible maps, but that would need a different approach: in order to support this claim, Section 4 will contain numerical computations on examples which are not covered by our analytical results, but which show  similar behaviors. We will comment further on these  issues in   Section 3 - Consequence 3.
\end{remark}

The random orbits $T^ix+\epsilon \xi_i$ generates a new random process when an observable is computed along them. Suppose that $w$ is a measurable real function defined on $\Omega$;  we     take $w(x)=-\log(||x-z||)$, where $z$ is a given point in $\Omega$. The process: $$ X_0=w(x + \epsilon \xi_0), X_1=w(Tx + \epsilon \xi_1),...  X_m=w(T^mx+\epsilon \xi_m)$$
is endowed with the probability $\mathbb{P}=\nu \times \theta^{\mathbb{N}}$  defined on the product space $\Omega\times S^{\mathbb{N}}$ with the product $\sigma$-algebra; a point in this space is the couple $(x,\ \overline{\xi}:=\{\xi_0,\xi_1,\cdots,\xi_m,\cdots \})\in X\times S^{\mathbb{N}}.$\footnote{With this final notation the random process is better defined as $T^ix+\epsilon \Pi_i(\overline{\xi}),$ where $\Pi_i(\overline{\xi})$ projects onto the $i$th component $\xi_i.$}

\begin{remark}Before checking conditions  $D_2(u_m)$ and $D'(u_m),$  we notice that, contrarily to the random setting studied in \cite{22}, the random variable $X_0$ depends now not only on the initial condition $x\in \Omega$, but also on the random variable $\xi$ and this makes $\mathbb{P}$ stationary.
\end{remark}
With the given choice of the observable $w$,  the set $U_m$ is explicitly given by
$$
U_m=\{(x,\xi); \ ||(x+\epsilon \xi)-z||\le e^{-u_m}\}
$$
 For convenience, we will also set  $V_m:=B(z,e^{-u_m})$, the ball of center $z$ and radius $e^{-u_m}.$\\
\begin{proposition}\label{P1}

Let us suppose that our dynamical systems verifies the Assumption M and it is perturbed with observational noise satisfying the Assumption N. Then conditions $D_2(u_m)$ and $D'(u_m)$ hold for the observable $w$.
\end{proposition}
\begin{proof}

We will give the proof when $T$ is continuous on the torus. The extension to the piecewise Lipschitz case is straightforward and it could be done as explained in the analogous extension   proofs of Propositions  4.2 and 4.5 in \cite{22} to which we defer for further details. We begin to check condition $D_2(u_m)$  by estimating the contribution given by the first term on the l.h.s of Eq.~\ref{D1}:

$$\p\left(X_0>u_m\cap
  \max\{X_{t},\ldots,X_{t+\ell-1}\leq u_m\}\right)= $$
  \begin{equation}\label{E1}\iint d\nu \ d \theta^{\mathbb{N}} {\bf 1}_{\{g(x+\epsilon\xi_0)>u_m\}}{\bf 1}_{\{g(T^tx+\epsilon\xi_t)\leq u_m\}}..{\bf 1}_{\{g(T^{t+l-1}x+\epsilon\xi_{t+l-1})\leq u_m\}},
  \end{equation}

the set of integration  variables here being $(x, \overline{\xi}^{(t+l)})$ with $\overline{\xi}^{(t+l)}=(\xi_0, \xi_1, \xi_{t+l-1})$. We apply Fubini's theorem and   factorize the integrals by exploiting the independence of the variables $\xi_l$, so that the previous expression becomes:

$$ \int d\nu \int {\bf 1}_{\{g(x+\epsilon\xi_0)>u_m\}} d\theta(\xi_0)  \prod_{i=1}^{l-1} \int {\bf 1}_{\{g(T^{t+i}x+\epsilon\xi_i)\leq u_m\}} d\theta(\xi_i). $$

Let us introduce the measurable functions :

$$H_m(l,x)= \prod_{i=1}^{l-1} \int {\bf 1}_{\{g(T^{i}x+\epsilon\xi_i)\leq u_m\}} d\theta(\xi_i) \ ;$$
$$G_m(x)=\int {\bf 1}_{\{g(x+\epsilon\xi_0)>u_m\}} d\theta(\xi_0) $$

with  $G_m(x) \in L^1_\nu$ and $H_m(l,x) \in \mathcal{B}.$ Then  Eq.~\ref{E1} can be rewritten as $\int d\nu \  G_m(x)H_m(l, T^tx)$.  By the decay of correlations assumption, we get

$$\left| \int d\nu \ G_m(x)\ H_m(T^t(x)) - \int d\nu \ G_m(x) \int d\nu \ H_m(l,x) \right|  $$
$$\leq C\ ||G_m||_{1}||H_m||_{\infty} \t^{-2}\le C \ t^{-2},$$

where $C$ is a constant depending only on $T$. In the previous equation, the second term on the l.h.s corresponds to $\p(X_0>u_m)  \p(M_{\ell}\leq u_m)$ which is the second term on the l.h.s. of the condition $D_2(u_m)$. Let us note, and this will be useful later, that $D_2(u_m)$ holds with $\gamma(m,t)=\gamma(t)=C^{*}t^{-2}$ for some $C^{*}>0$ and $t_m=m^{-\beta}$, with $1/2<\beta<1.$

In order to deal with Condition  $D'(u_m)$, we follow the same strategy as in \cite{22}. We begin
to define the {\em approximated first return time of the point} $x$ in $V_m$ in the following way: we fix the couple   $(\xi, \xi')\in S^2$ and we set
 $$r_{V_m, \xi,\xi'}(x):=\min\{j\ge1, \ T^jx+\epsilon \xi'\in V_m; x+\epsilon \xi\in V_m\}.$$  Notice that we keep fixed  the variables $\xi, \xi'$ while   iterating the point $x$. Moreover, instead of $x$, we require the initial condition $x+\eps\xi$ to be in $V_m$. Then we
 define the  {\em approximated first return time  of the set} $V_m$ into itself as:
  $$R_{V_m, \xi,\xi'}:=\min_{\{x, x+\epsilon \xi \in V_m\}} \{r_{V_m, \xi, \xi'}(x)\}.$$ We observe that
  $||T^j(V_m-\epsilon \xi)+\epsilon \xi'||\le \eta^j||V_m-\epsilon \xi||,$   where $\eta>1$ is the highest rate of separations for the  points. We use the symbol $||\cdot||$ indifferently to denote {\em distance} and {\em diameter}. The notation $z+B$, where $z\in M$ and $B$ is a subset of $M$,  stands for  the set $\cup_{w\in B}{\{z+w\}}$). We now fix some sequences $(\alpha_m)_{m\in \mathbb{N}}$  going to infinity and such that $\alpha_m=o(\log k_m)$, where $k_m$ is the sequence defined in (\ref{def:Un}). Therefore, whenever
$$||T^j(z-\epsilon \xi) + \epsilon \xi' - (z-\epsilon \xi)|| > 2\eta^j||V_m-\epsilon \xi||, $$  $\forall j=1,\cdots,\alpha_m$, then $$T^j(V_m-\epsilon \xi) + \epsilon \xi' \cap( V_m-\epsilon \xi) = \emptyset,$$  which in turn implies that $R_{V_m, \xi,\xi'}>\alpha_m$. Since
$$\{(\xi, \xi')\in S^2; R_{V_m, \xi,\xi'}\le \alpha_m\}\subset \cup_{j=1}^{\alpha_m}\{(\xi, \xi'\in S^2); ||T^j(z-\epsilon \xi) + \epsilon \xi' - (z-\epsilon \xi)||\le 2\eta^j||V_m-\epsilon \xi|| \},$$ we have $$\theta^2\{(\xi, \xi')\in S^2; R_{V_m}(\xi, \xi')\le \alpha_m\}\le
$$
$$\sum_{j=1}^{\alpha_m}\int d\xi \rho(\xi) \int d\xi' \rho(\xi')  {\bf 1}_{\{\xi', ||\xi'-\frac{(z-\epsilon \xi-T^j(z+\epsilon \xi))}{\epsilon}||\le \frac{2 \eta^j}{\epsilon}||V_m-\epsilon \xi||\}}\le $$
$$
 \mathcal{ O}\left(||\rho||_{\infty}\sum_{j=1}^{\alpha_m}2^d\ \eta^{jd}\ ||V_m||^d\epsilon^{-d}\right)\le \mathcal{O}\left(||g
||_{\infty}||V_m||^d\epsilon^{-d}\eta^{d\alpha_m}\right)
$$
where "$\mathcal{O}$" takes into account the multiplicative factor given by the volume $K_d$ of the unit hypersphere $S$.
We now have:
$$
m\sum_{j=1}^{[\frac{m}{k_m}]}(\nu\times\theta^{\mathbb{N}})\{(x,\overline{\xi}); X_0>u_m; \ X_j>u_m\}=
$$
$$
m\sum_{j=\alpha_m}^{[\frac{m}{k_m}]}\int d\nu \{\int d\theta(\xi) {\bf 1}_{\{g(x+\epsilon \xi)>u_m\}}\int d\theta(\xi') {\bf 1}_{\{g(T^jx+\epsilon \xi')>u_m\}}\}+
$$
$$
m\sum_{j=1}^{\alpha_m}\int \int {\bf 1}_{\{(\xi, \xi'); R_{V_m, \xi,\xi'}\le \alpha_m\}}d\theta(\xi)  d\theta(\xi') \cdot A(j,\xi,\xi')= I +II $$ where $A(j,\xi,\xi')=\left\{\int d\nu {\bf 1}_{\{x+\epsilon \xi\in V_m\}}\ {\bf 1}_{\{T^jx+\epsilon \xi'\in V_m\}}\right\}.
$

Again, the first term (I) on the l.h.s. can be estimated  by using decay of correlations applied to the (same) observable $\tilde{H}(x)=\int d\theta(\xi)  {\bf 1}_{\{g(x+\epsilon \xi)>u_m\}};$ we easily get:
$$
I \ \le m\sum_{j=\alpha_m}^{[\frac{m}{k_m}]} \{\mathbb{P}(U_m)^2+ C \mathbb{P}(U_m)j^{-2}\}
\le \frac{(n\mathbb{P}(U_m))^2}{k_m}+mC\mathbb{P}(U_m)\sum_{j=\alpha_m}^{[\frac{m}{k_m}]}j^{-2}=
{\mathcal O}(\frac{\tau^2}{k_m}+\tau \sum_{j=\alpha_m}^{[\frac{m}{k_m}]}j^{-2})\underset{m\to+\infty}{\longrightarrow} 0
$$
since $m\mathbb{P}(U_m)\rightarrow \tau.$ \\ For the second term (II) we use Holder's inequality and the fact that $\nu$ is equivalent  to $\L$  \footnote{We will use the symbol ``$\approx$'' to signify that equivalence, namely  there exists a positive constant $\iota$ such that for any measurable set $A$ we have  that $\iota^{-1}\L (A)\le \nu(A)\le \iota \L(A)$.} with essentially bounded density and $\L$ is translationally invariant:
%(by using again the facts that $\nu(V_n)\sim \tau/n$,  $\nu$ %is equivalent to $m$ with an essentially bounded density and %finally by (\ref{def:Un})):

$$
II\le {\mathcal O}\left(m\sum_{j=1}^{\alpha_m} \L(V_m)\theta^2\{(\xi, \xi')\in S^2; R_{V_m, \xi,\xi'}\le \alpha_m\}\right)
$$

$$
\le{\mathcal O}\left(m\sum_{j=1}^{\alpha_m} \mbox{\L}(V_m) ||\rho
||_{\infty}||V_m||^d\epsilon^{-d}\eta^{d\alpha_m}\right)
$$
%= \mathcal{O}(  \frac{n^2}{k_n} \mathbb{P}(U_n)^2  \ %||g|_{\infty} %\frac{\eta}{1-\eta}\eta^{\alpha_n}\epsilon^{-1})=$$
%$$
%\mathcal{O}(\frac{\eta^{\alpha_n}}{k_n})\underset{n\to+\infty}%{\longrightarrow} 0
%$$
%since $\alpha_n=o(\log k_n),$ and we used the fact that
Since $\mathbb{P}(U_m)\approx \L(V_m)$ and $||V_m|| \approx  \L(V_m)^{1/d}$ we finally  have
$$
II={\mathcal O}(\frac{m^2}{k_m} \ \mathbb{P}(U_m)^2 \  \eta^{d\alpha_m})={\mathcal O}\left(\tau^2 \ \frac{\eta^{d\alpha_m}}{k_m}\right),
$$
which goes to zero with the prescribed assumptions for $\alpha_m$ and $k_m.$\\
\end{proof}
\section{Generalizations and consequences}

In this section  we would like  to point out a few  interesting properties of the observational noise.  We start by an explicit calculation of the quantity $\tau$ defined in (\ref{eq:un}) for the observable  (\ref{g1}). The computation  is done in $d=1$, but the generalization to higher dimensions  is trivial. As we have anticipated in Section 1, the following proposition   requires only the existence of the Gumbel law for the process under investigation, which is of course true for   systems verifying Proposition \ref{P1}.
 \begin{proposition}\label{P2}
 Let us suppose that the one-dimensional dynamical systems $(\Omega, \mathcal{B}, \nu, T)$ verifies a Gumbel law for the process $X_m(x,\overline{\xi}) :=-\log(|T^mx+\Pi_m(\overline{\xi})-z|)$ endowed with the probability $\mathbb{P}=\nu\times \theta^{\mathbb{N}},$ and also that $\theta$ is the Lebesgue measure measure on $S$. Then the linear sequence $u_m:=u/a_m + b_m$, defined by \ref{eq:un}, verifies:
 $$
 a_m=1; \ b_m=\log\left(\frac{m \nu(B(z, \eps))}{\varepsilon}\right)
$$
 \end{proposition}
 \begin{proof}
 We begin to observe that
 $$
m\ \mathbb{P}(X_0>u_m)= m \ \int d\nu(x) \ \left( \int d\theta(\xi)  \ {\bf 1}_{\{\xi; \ ||\xi-\frac{(z-x)}{\varepsilon}||<\frac{e^{-u_m}}{\varepsilon}\}}\right)=
$$
$$
m \ \int  d\nu(x) \theta \left(B\left(\frac{z-x}{\varepsilon}, \ \frac{e^{-u_m}}{\varepsilon}\right)\cap B(0,1)\right)
$$
since  the variable $\xi$ must stay in the ball of center $0$ and radius $1$.\\
Let us now introduce  $$ u_m:= -\log (\frac{ \varepsilon \ \tau}{m \nu(B(z, \eps))})=u/a_m+b_m,$$ with $u:= -\log \tau; \ a_m:=1; \ b_m:=\log(\frac{m \nu(B(z, \eps))}{\varepsilon})$;   and observe that
$$
m \ \theta \left(B\left(\frac{z-x}{\varepsilon}, \ \frac{e^{-u_m}}{\varepsilon}\right)\cap B(0,1)\right)\le m \frac{e^{-u_m}}{\varepsilon}\le \frac{m}{\varepsilon}\frac{ \varepsilon \ \tau}{m \nu(B(z, \eps))}.
$$
This bound is independent from $m$ and integrable\footnote{Here we use crucially the fact that $\theta$ is {\em exactly} Lebesgue, since its translational invariance property allows us to get rid of the variable $x$ in the center of the ball.}. \\
We can apply the theorem  of   dominated convergence since
$$
 \lim_{m\rightarrow \infty} m \ \theta \left(B\left(\frac{z-x}{\varepsilon}, \ \frac{e^{-u_m}}{\varepsilon}\right)\cap B(0,1)\right)= {\bf 1}_{B(z,\varepsilon)}(x) \frac{ \ \tau}{ \nu(B(z, \eps))}
$$
Having passed the limit inside, the integral finally gives $\tau$:
$$
m\ \mathbb{P}(X_0> u +b_m)\rightarrow \tau
$$
and therefore
$$
\mathbb{P}(M_m\le u+b_m)\rightarrow \exp(-e^{-u})
$$
which is exactly the {\em Gumbel law}.\\Whenever $d>1$  a similar computation  immediately gives  that the linear sequence $u_m:=u/a_m + b_m$ verifies:
 $$
 a_m=d; \ b_m=\frac1d\ \log \left(\frac{K_d\ m \nu(B(z, \eps))}{\varepsilon^d}\right).$$
\end{proof}

The following useful consequences will be exploited in the next section:
\begin{itemize}
\item {\em Consequence 1}\\
The scaling parameter $b_m$ depends on the target point $z$   via  the local density of the invariant measure in a ball of radius   given by the    $\eps$. Let us start by considering the case of absolutely continuous  invariant measures.  If the point $z$ is visited with less frequency, the local density will be of lower order in  $\eps$, which means that one should go to higher values of $m$ in order  to have a reliable statistics.  This is the case, for instance,  for the points $\pm1$ for the map introduced by Hemmer  \cite{Hemmer}:

\begin{equation}
T(x)=1-2\sqrt{|x|}
\label{hemmereq}
\end{equation}

and defined on the interval $[-1, 1]$ . The invariant density $\rho$ can be computed directly
 by inspection and reads: $\rho(x)=\frac12(1-x)$.
 %In $x=-1$  the density of the absolutely continuous invariant measure $\nu$ % behaves like $x$ .
  Therefore $\nu(B(-1, \varepsilon))\approx \varepsilon^{2}$ and $b_m\approx \ \log(m \varepsilon).$  \\  A complementary issue will appear whenever the map exhibits a laminar behavior in some regions of the phase space and therefore it will  spend there a lot of time. This happens, for instance, for  the well-known map of Pomeau-Manneville \cite{PM}, which could be written as:

\begin{equation}
\begin{cases}
T_1(x)= x + 2^{\alpha}x^{1+\alpha}, \ 0\le x \le 1/2\\
T_2(x)=  T(x)= 2x-1, \ 1/2 \le x \le 1\\
\end{cases}
\label{PM}
\end{equation}

 The origin $0$ is a neutral fixed point and, for $0<\alpha<1$, the density of the absolutely continuous invariant measure $\nu$ behaves like $x^{-\alpha}$ in the neighborhood of $0$ \cite{LSV}. Therefore, $\nu(B(0, \varepsilon))\approx \varepsilon^{1-\alpha}$ and $b_m\approx \ \log(m \varepsilon^{-\alpha}).$  We will analyze in the next section which finite size effects arise for these two examples.\\

 {\em Warning}: strictly speaking the previous two maps do not fit with the assumptions of Proposition \ref{P1}. In fact, in both cases it is not possible to prove a polynomial decay of correlations against {\em all} $L^1_{\nu}$  functions as it was shown in \cite{AFLV}. On the other hand, an inspection of the proof shows that  we only require the $L^1_{\nu}$ property for characteristic functions, so that, in principle, such a decay could be obtained. This was achieved, for instance, in the case of rotations by using a Fourier series technique \cite{31a}. An additional problem concerns the rate of decay. For the Pomeau-Manneville map it is of order $n^{-\frac{1}{\alpha}+1}$  \cite{LSV}, so that it fits with our assumption whenever $\alpha< 1/3$; for the Hemmer's map the situation is  worst since the correlations decay as $\frac1n$  \cite{CHMV}. Nevertheless, we conjecture that Proposition \ref{P1} could be applied to the latter case as well, as we will argue in {\em Consequence 3}, and  as we will show numerically in the next section.

\item {\em Consequence 2}\\Let us observe that, if we define the linear scaling factor as above, we have convergence towards the Gumbel law $e^{-e^{-u}}$ for {\em any} $z$. This shows that the extremal index  defined and studied  in \cite{15,22} is $1$ everywhere implying that there are no points which behave like unstable fixed points of deterministic dynamical systems.
\item {\em Consequence 3}\\
Although we were able to prove Proposition \ref{P1} whenever  the invariant measure is equivalent to Lebesgue, we pointed out that Proposition \ref{P2} is true for any invariant  measure, providing that one can show the existence of an EVL.  Let us therefore suppose that a given dynamical system with invariant measure $\mu$, not necessarily equivalent to Lebesgue, admits an extreme value law for the observable $w$ under the  observational noise. Then we can apply Proposition \ref{P2}, which only requires that $\theta$ is Lebesgue. Remember that in the expression for $b_m$ we are considering a fixed positive size for the noise $\eps$. Let us suppose that at this scale we have
 $\nu(B(z, \eps))\approx \eps^D$, where $D$ is usually an estimation of the geometric and fractal properties of $\nu$ at the point $z$ \footnote{We give an example:
a measure $\nu$ on $\R^+$ is called {\em Ahlfohrs upper semi-regular} if there is a constant $C>0$ and a real number $\alpha>0$ such that for all non-empty open intervals $I\subset{\R^+}$
$
\frac{\nu (I)}{(\mbox{diam}(I))^\alpha}<C.
$\\
Ahlfors upper semi-regular measures include fractal measures like the measures of maximal dimension of dynamically defined Cantor sets, i.e. Cantor sets that arise from smooth expanding repellors. More generally, given an invariant measure $\mu$ and whenever the limit $\lim_{r\rightarrow 0^+}\frac{\log\mu(B(x,r))}{\log r}$ exists $x$-$\mu$-almost everywhere, then this limit equals the {\em Hausdorff dimension of the measure} $\mu$, $HD(\mu):=\inf \{\mbox{Hausdorff dimension of}\ Y, \mu(Y)=1\}$ \cite{YY}, also called {\em information dimension}.  In some cases $HD(\mu)$ is given by  suitable relations between Lyapunov exponents and entropies, formulae better known as {\em Kaplan-Yorke} and {\em Ledrappier-Young}: see \cite{RE} for a detailed exposition of these issues. }. In this case, if the ambient space has dimension $d$,  the linear scaling parameter $b_m$ has the form:
\begin{equation}
b_m \sim \frac1d \ \log(m \eps^{D-d}).
\label{superformula}
\end{equation}
Therefore, we have  an useful technique to detect the local dimensions of the measure, with a {\em finite} resolution given by the strength of the noise. This will also allow us to compute  directly the distribution of the maxima with the linearization, given the explicit expression of the $u_m$. This would be particularly useful whenever the invariant measure is singular and therefore the GEV distribution does not admit a probability density function: see Section 3.1 in \cite{30} for a detailed discussion on this point.
\item {\em Consequence 4}\\ As we said in the Introduction, another interesting property of the observational noise is its direct relationship with the statistics of first hitting times in small sets.
We first define the {\em first hitting time of the set $V_m$} in a slightly different manner which respect to the quantity $r_{V_m}$ introduced above. Given the triple $(V_m, \ x\in M, \ \overline{\xi}\in S^{\mathbb N})$, we set:
$$
{\mathcal R}_{V_m}(x,\overline{\xi}):=\min\{j\ge 1, \ T^jx+\epsilon \xi_j\in V_m\}
$$
Let us notice that, contrarily to $r_{V_m}$, the error increment changes at each step since we are now dealing with a {\em true} random orbit; this easily implies that

$$
\mathbb{P}(M_m\le u_m)=\ \mathbb{P}({\mathcal R}_{V_m}>m).
$$

By using the expression of $u_m$ found above and setting consequently $V_m=B(z; \frac{\varepsilon \tau}{m \nu(B(z, \eps))})$, in dimension $1$,  we could rewrite the previous formula as, for $t\in \mathbb{R}:$
\begin{equation}
\mathbb{P} \left(\frac{t}{m}{\mathcal R}_{B(z; \frac{t}{m})}>t\right)\rightarrow e^{-\frac{t \nu(B(z, \eps))}{\varepsilon}}
\label{hts}
\end{equation}

which shows that the first hitting time follows an exponential law tempered by the strength noise $\varepsilon.$

\end{itemize}

\section{Discussion and numerical results}

In this section we discuss some important implications connected to the introduction of observational noise in finite time series. In particular, through numerical experiments devised on low dimensional maps, we show that the influence of truncations and observational noise is related to the intensity of the perturbation applied. Some of the theoretical findings presented in the previous sections present a practical interest in a wide range of applications namely the analysis of the role of truncation errors for instrument with low accuracy,  the statistics of points visited sporadically in the analysis of recurrence of time series and   the possibility of computing attractor dimension by using Eq. \ref{superformula} as an alternative way to other techniques. For each maps, we present and comment numerical experiments and outline possible further applications.\\ Before introducing such examples, let us make a useful comment. A close
 inspection to the proof of Proposition \ref{P1}, shows that the parameter $\eps$ appears in the denominator of one factor in the r.h.s. of the term (II) at the end of the proof. This means that the convergence gets better when $\eps$ is large, which is not surprising since a large value of the perturbation implies a more stochastic independence of the process. On the other hand, if we want to use the form of the linear scaling parameter $b_m$ to catch the local properties of the invariant measure $\nu$, we need small values of $\epsilon$. A judicious  balance of the value of $\eps$ between these two regimes is therefore necessary when we pursue such numerical analysis.

\subsection{Truncation}
We  perturb a ternary shift  map with truncation error on the different digits $q$:

$$\varphi(i)=\operatorname{trunc}( 3x_i\ \mbox{mod}\ 1 , q)$$

The experiment consists in producing 30 orbits starting from different initial conditions taken  on the support of the  truncated measure. The length of the orbit is fixed according to the results presented in \cite{21} to be  such that $n=1000$. In order to analyze the effect of varying the bin length combined to the order of the truncation, we consider three different values of $m$. Each series of $w$  is therefore divided in $n$ bins and in each of them the maxima $M_j$ are extracted and then fitted to the GEV model via the L-moments procedure described in \cite{31a}. Note also that  the points $z$ are chosen after applying the truncation.   The results are shown in fig \ref{trunc} where the behavior of  the shape parameter $\kappa$ is  compared to the asymptotic Gumbel law $\kappa=0$.  In the deterministic limit $q\gg1$ the usual Gumbel EVL is recovered, whereas for $q\to 1$ the discretization of the invariant measure becomes relevant and the asymptotic EVL appears as a collection of Dirac deltas thus producing a divergence of the shape parameter from $0$. This is exactly what is visible in Fig \ref{trunc}: the convergence gets worse when $q<6$ and at $q=3$ we are already unable to fit the GEV distribution for all the $z$ points considered, although by increasing the bin length $m$ the convergence improves as  one would expect. These results, as we tested, are reproducible in other maps and gives a more general indication that computing asymptotic properties on truncated series   leads to estimation errors and divergence even at high order of truncation. \\

\subsection{Highly recurrent and sporadic points}

With the  introduction of the  observational noise, the scaling parameter $b_m$  depends on the target point $z$ {\em via} the local density of the invariant measure in a ball whose radius is given by the  error $\eps$.   In {\em Consequence 1}  we said that if the point $z$ is visited with less frequency, the local density is of lower order with respect to $\eps$, which means that one should go to higher values of $m$ in order  to have a reliable statistics. Here we want to test  that the order of $m$ needed to get convergence to the asymptotic $b_m$ is lower for a highly recurrent point then for a sporadic one. As  highly recurrent point  we choose $z=0$ for the Pomeau-Manneville map in Eq. \ref{PM} and as sporadic one the point $1$ of the map introduced by Hemmer and reported in Eq. \ref{hemmereq}. The experiment consists in computing $30$ realizations of the maps perturbed with  observational noise. Again, we  fit the maxima of the observable $w$ to the GEV distribution  by using the $L$-moments procedure and compare the values of $b_m$ obtained experimentally to the theoretical ones stated in Proposition \ref{P2}.  We report here the results for three different bin lengths $m=1000,10000,30000$  in Fig. \ref{PMfig}  for the Pomeau Manneville map and in Fig. \ref{Hemmerfig} for Hemmer map. The figures show how $b_m$ varies as a function of the noise $\epsilon=10^{-p}$, in terms of $p$.   In both  cases we observe convergence towards the theoretical values (solid lines) for high values of $\epsilon$ (low orders in $p$) whereas  in the limit of weak noise  one must increase the bin lengths to get convergence. The main result to be highlighted here is the better convergence of  highly recurrent points with respect to the ones visited  sporadically. This important property can be used to study time series recurrences  and  identify extremes as the points visited rarely for which the convergence towards the asymptotic parameters is bad.   The main advantage of studying recurrence properties in this way over applying other techniques  is due to the built-in test of convergence of this method: even for a point rarely recurrent there will be a time scale $m$ such that the fit converges. For smaller $m$, we can therefore  such a $z$ as a \textit{sporadically recurrent} point of the orbit as explained in \cite{31b}.  There  we show how to use this property to define rigorous recurrences in long temperature records collected at several weather stations. The convergence or divergence of the fit  allows us  for discriminating between temperatures belonging to the normal variability associated to the time scales defined by the bin length (e.g. the seasonal cycle)  or as anormal temperature if there are no or few recurrences in $m$.

\subsection{Attractor dimensions}

As we have already said in {\em Consequence 3},  if the invariant measure is not absolutely continuous, one could still perform the previous analysis, but both the ambient space dimension $d$ and the local dimension $D$ will enter in the computation of $b_m$ via Eq. \ref{superformula}.  This formula can be used in principle to test whether a map has a fractal support by comparing the  local dimensions with the  the ambient space dimension: it is enough to check how the obtained $b_m$ depend on the intensity of the noise $\epsilon=10^{-p}$. We should point out that, as we said in the footnote 1,  we are targeting the Hausdorff dimension of the measure. We test this idea  on the classical Iterated Function System $\{ T_1, T_2 \}$ used to produce  a Cantor set:
\begin{equation}
\begin{cases}
T_1(x)= x/3 \mbox{ with weight } q_1\\
T_2(x)=(x +2)/3 \mbox{ with weight } q_2\\
\end{cases}
\label{IFS_C}
\end{equation}
where $x \in [0,1]$, and we set $q_1=q_2=1/2$. Therefore, at each time step, we have the same probability to iterate $T_1(x)$ or $T_2(x)$; the {\em balanced} invariant measure associated to the map and the way to construct the orbit to detect the maxima are described in \cite{30}.  The results for different $m$ are reported in Fig. \ref{BM} for an  average among $30$ different realizations and three different bin lengths $n=1000,\ m=300,1000,3000$. Experimental data follow the prediction of Eq~\ref{superformula}  with $D=\log(2)/\log(3)$  as   Hausdorff dimension and $d=1$ as  ambient  space dimension but only  up to a certain noise intensity $p$ beyond which they reach a {\em plateau}. A justification for this behavior is that when the noise intensity is very small, the system needs longer trajectories - higher $m$ - to explore the ball of radius $\epsilon$. This gives an implicit criterion for the selection of the bin length $m$ needed to observe reliable results and it says that one should be careful in applications where the intensity of the observational noise is small compared to the scale of the dynamics. This analysis is confirmed by the results obtained for the Lozi map:
\begin{equation}
\begin{array}{lcl}
x^{(1)}_{t+1}&=&x^{(2)}_t +1 -a |x^{(1)}_t|\\
x^{(2)}_{t+1}&=&b x^{(1)}_t\\
\end{array}
\label{lozi}
\end{equation}
for which we consider the classical set of parameter   $a=1.7$ and $b=0.5$. Young \cite{young1985bowen} proved the existence of the SRB measure for the Lozi map and found the value $D=1.40419$ for the Hausdorff  dimension of the measure     by computing the Lyapunov exponents  and using a Kaplan-Yorke like formula.   The experiments is exactly the same described for the Cantor set and the results are presented in Fig.~\ref{lozifig} for the values of  $m=1000,3000,30000$. Agreement with the theoretical behavior of $b_m$, as expected from Eq. \ref{superformula}   represented in Fig.~\ref{lozifig} by the solid straight lines, is found for small   values of $p$ which means for large $\epsilon$.  For the Lozi map the convergence is worse then for the Cantor case.  This phenomenon has been already observed in \cite{30} and, up to now, there is not a clear explanation for it. However, by scanning numerically the ($m,p$) space, one can infer  the intervals of such parameters such that the $b_m$  converge towards the asymptotic results:  for example, when $m=30000$, an order of noise intensity $p\leq 3$  is needed to get convergence to the predicted theoretical values. Since there is only a limited range of $p$ such that the $b_m$  convergence to the prediction of Eq. \ref{superformula},  one has to take extremely care on using the results obtained with this method  to estimate   fractal dimensions. A good strategy to overcome this problem is to discard the value of $b_m$ which show no dependence on $p$ and check that the remaining points are sufficient  to perform a linear fit of $b_m$ vs $p$.\\

The possibility of computing fractal dimensions by using   a random  perturbation of a  dynamical system is an interesting  fact: when random perturbations are applied, one usually expects the orbit to explore asymptotically the ambient space and this normally hide the fractal property of the measure. Instead, as we said at the beginning of this section,  for  reasonable choices of $m$ and $p$, one gets   accurate information on  the fractal dimension by fitting the $D$ parameter of  Eq. \ref{superformula}. This finding opens new questions, for instance if it is possible to get analogous formulas in the case of random transformations. An application can be to adapt this theory  for  multi-scale systems whose description is usually made through an approximation of a dynamic consisting of a deterministic and a stochastic components. For example, by tuning the strength of the stochastic components,  one can study how the noise affects the structure of the deterministic dynamics; these aspects will be explored in forthcoming works.
%% Enter the largest bibliography number in the facing curly brackets
%% following \begin{thebibliography}

\medskip\noindent\textbf{Acknowledgements}.
SV was supported by the ANR-
Project {\em Perturbations}, by
the PICS ( Projet International de Coop\'eration Scientifique), Propri\'et\'es statistiques des sys\`emes dynamiques det\'erministes et al\'eatoires, with the University of Houston, n. PICS05968 and by the projet MODE TER COM supported by {\em R\'egion PACA, France}.  SV thanks support from  FCT (Portugal) project PTDC/MAT/120346/2010. The authors aknowledge Jean-Ren\'e Chazottes,  Jorge Milhazes Freitas and Paul Manneville for useful discussions. DF and SV acknowledge the
Newton Institute in Cambridge where this work was completed during
the program {\em Mathematics for the Fluid Earth}.

\begin{figure}[ht]
\centering
\includegraphics[width=120mm]{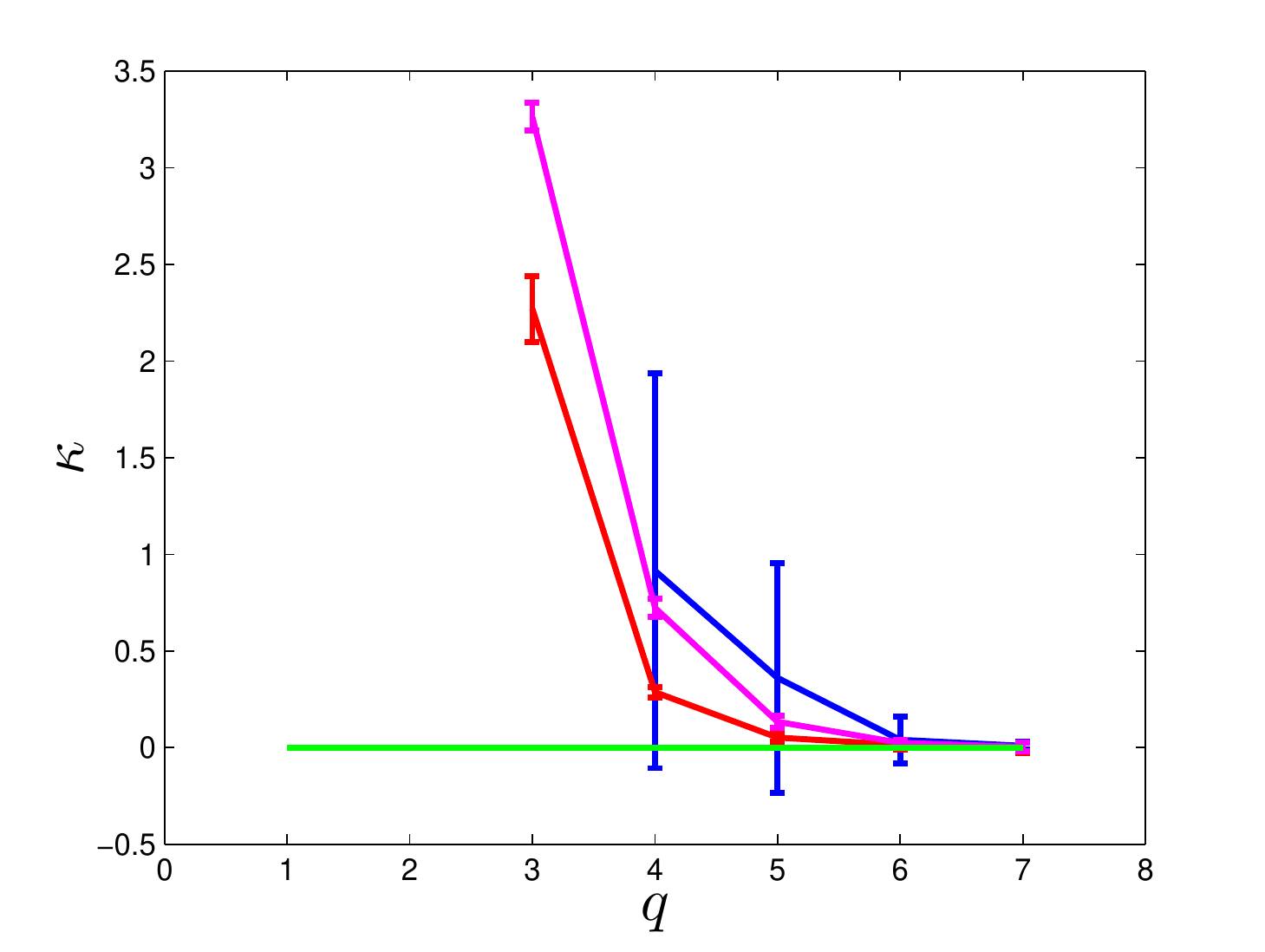}
\caption{Shape parameter $\kappa$ vs $q$, the digit where the truncation has been applied. Error-bars display the average of $\kappa$ over 30 realizations and the standard deviation of the samples for  $m$=300 (blue), $m$=1000 (magenta) and $m$=3000 (red). The green line correspond to the Gumbel law $(\kappa=0)$. $ n=1000$ for all the cases considered. }
\label{trunc}
\end{figure}

\begin{figure}[ht]
\centering
\includegraphics[width=100mm]{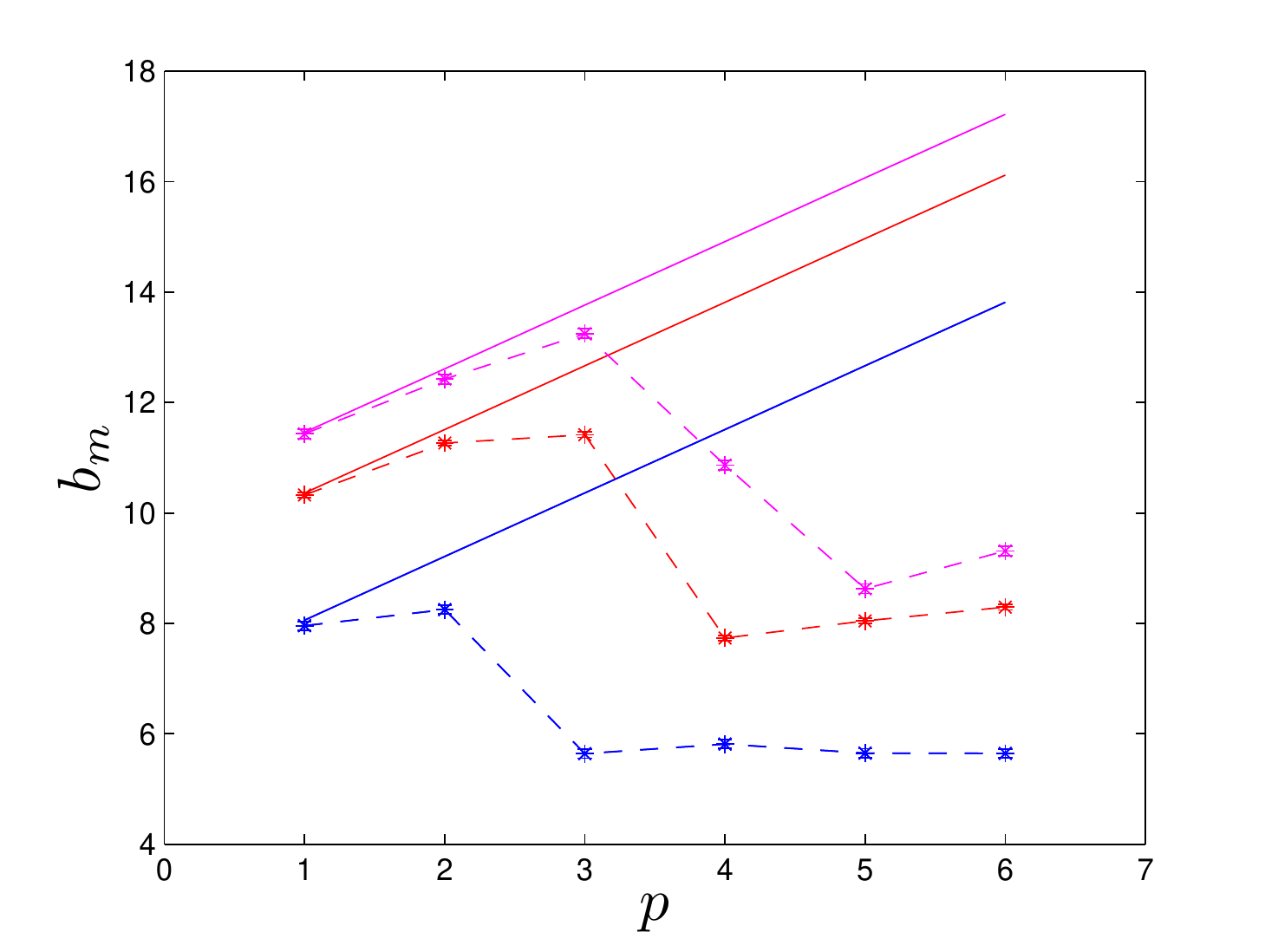}
\caption{Normalizing sequence $b_m$ vs intensity of the noise in terms of $p$ (we recall that $\epsilon=10^{-p}$) for the Pomeau Manneville map (Eq. \ref{PM}). We recall that Dashed error-bars display the average of $b_m$ over 30 realizations and the standard deviation of the sample. Solid lines the theoretical values. the blue, red and magenta curves respectively refers to $m=1000,10000,30000$, $z=0$. $ n=1000$ for all the cases considered.}
\label{PMfig}
\end{figure}

\begin{figure}[ht]
\centering
\includegraphics[width=100mm]{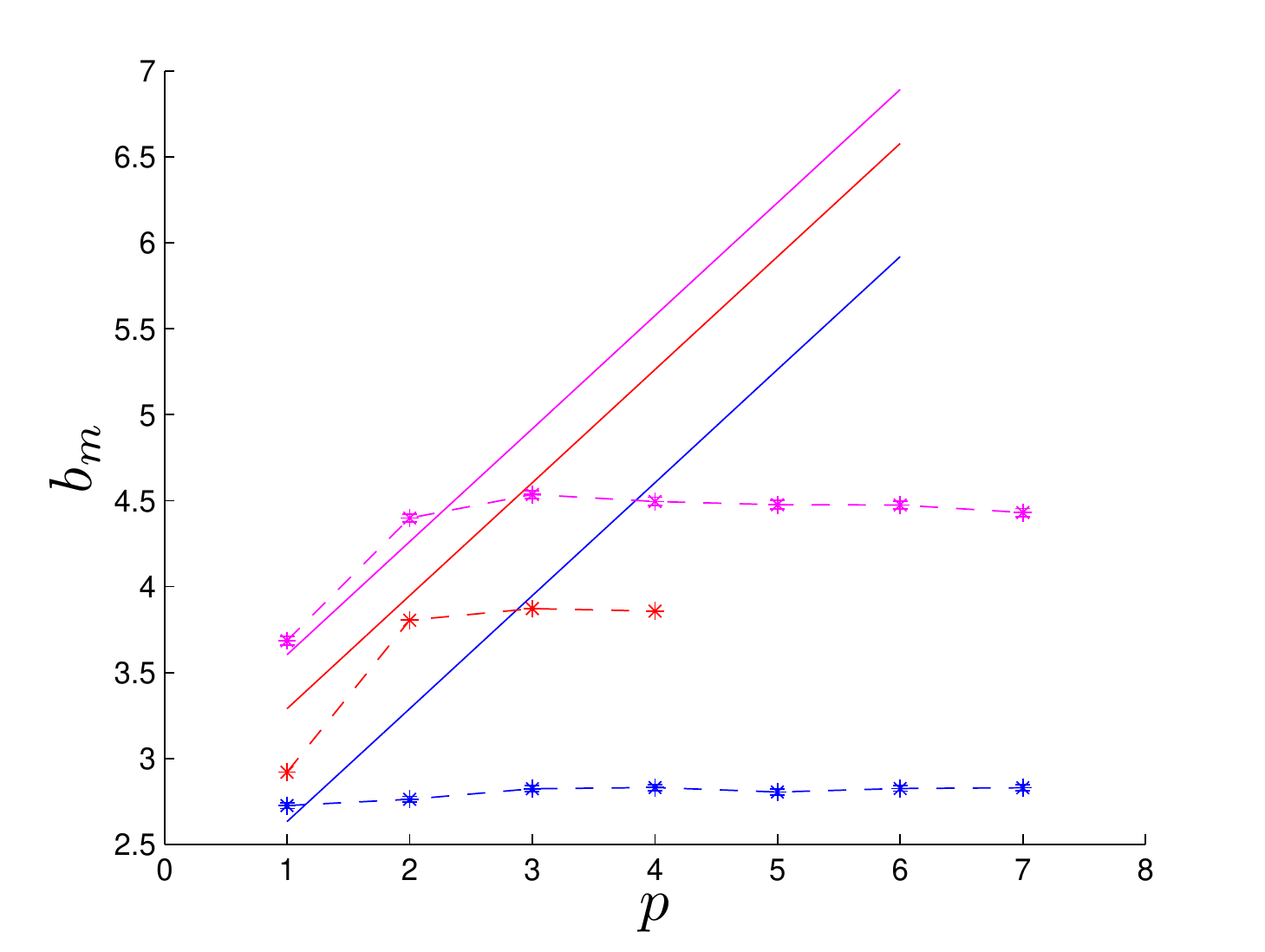}
\caption{Normalizing sequence $b_m$ vs intensity of the noise in terms of $p$ (we recall that $\epsilon=10^{-p}$) for the Hemmer map (Eq. \ref{hemmereq}). Dashed error-bars display the average of $b_m$ over 30 realizations and the standard deviation of the sample. Solid lines the theoretical values. the blue, red and magenta curves respectively refers to $m=1000,10000,30000$, $z=1$. $ n=1000$ for all the cases considered.}
\label{Hemmerfig}
\end{figure}

\begin{figure}[ht]
\centering
\includegraphics[width=100mm]{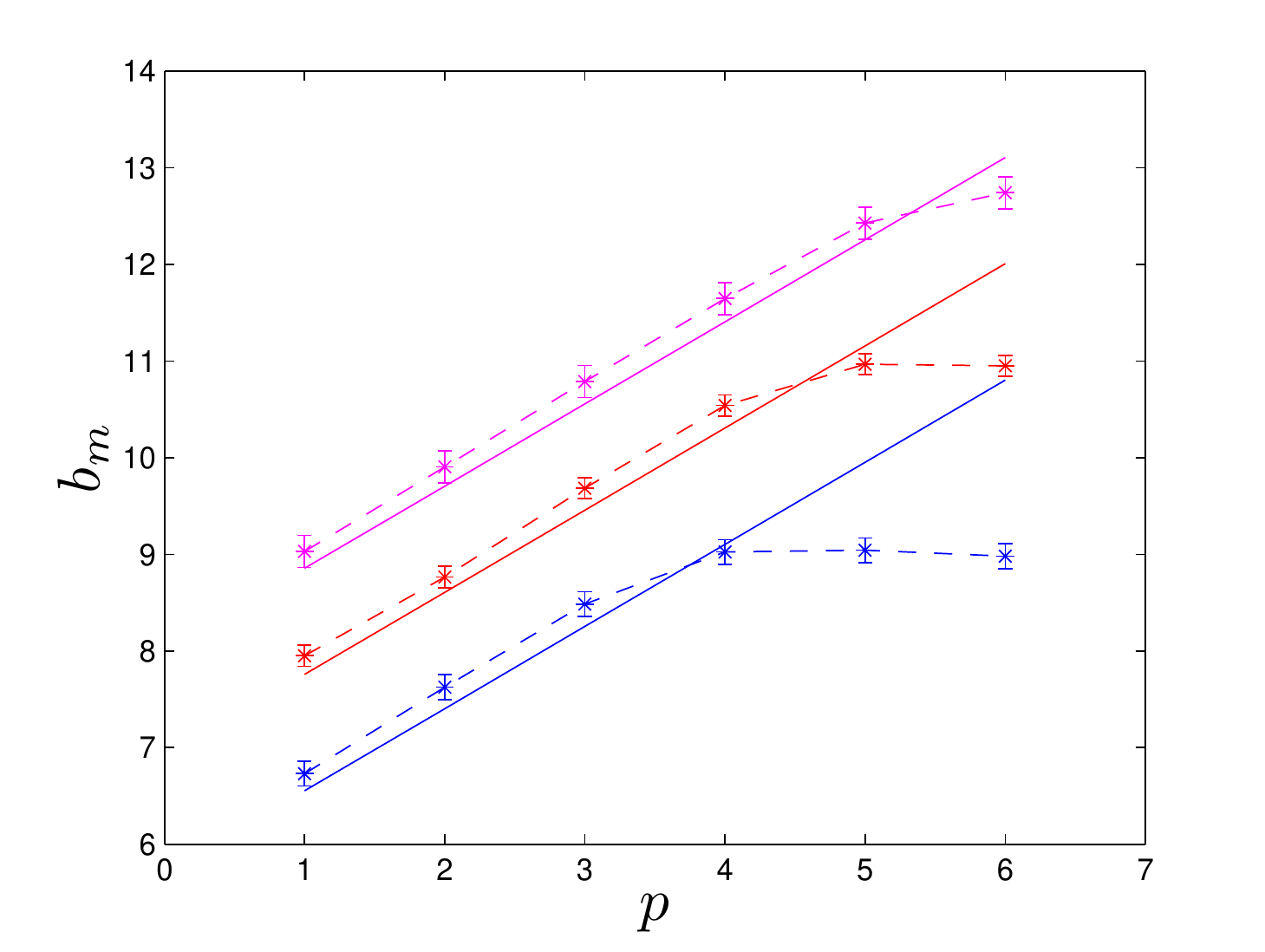}
\caption{Normalizing sequence $b_m$ vs intensity of the noise in terms of $p$  (we recall that $\epsilon=10^{-p}$) for the Cantor IFS (Eq. \ref{IFS_C}). Dashed errorbars display the average of $b_m$ over 30 realizations and the standard deviation of the sample. Solid lines the theoretical values. the blue, red and magenta curves respectively refers to $m=1000,10000,30000$, $z$s randomly chosen on the attractor. $ n=1000$ for all the cases considered.}
\label{BM}
\end{figure}

\begin{figure}[ht]
\centering
\includegraphics[width=100mm]{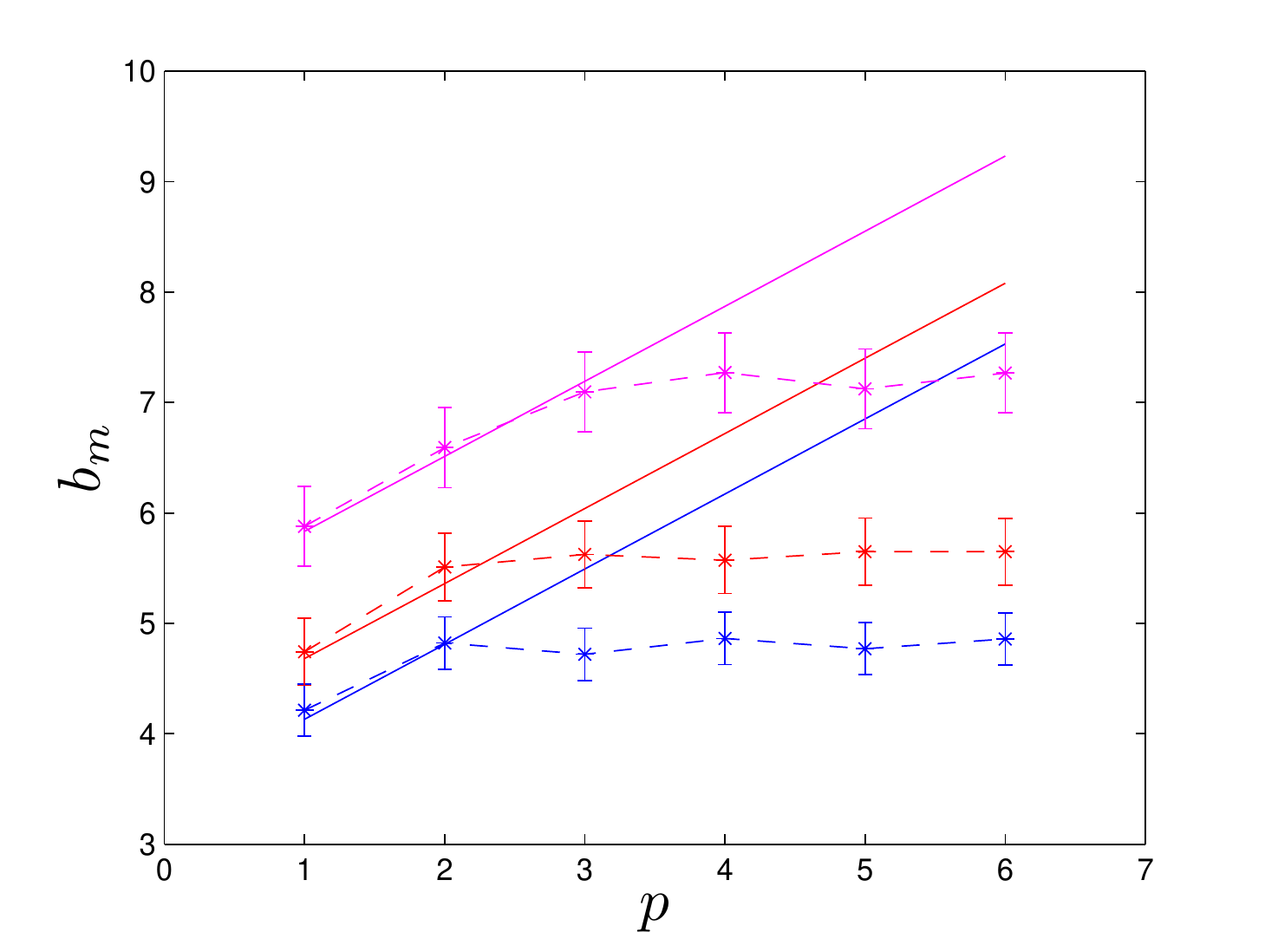}
\caption{Normalizing sequence $b_m$ vs intensity of the noise in terms of $p$  (we recall that $\epsilon=10^{-p}$) for the Lozi map  (Eq. \ref{lozi}). Dashed errorbars display the average of $b_m$ over 30 realizations and the standard deviation of the sample. Solid lines the theoretical values. the blue, red and magenta curves respectively refers to $m=1000,10000,30000$. The points $z$ are randomly chosen on the attractor. $ n=1000$ for all the cases considered}
\label{lozifig}
\end{figure}

\end{document}